\documentclass{amsart}
\usepackage{a4,amssymb,amsfonts,amsmath,enumerate,color,url,hyperref,mathbbol}

\def\:{\thinspace:\thinspace}
\numberwithin{equation}{section}
\newtheorem{theo}{Theorem}

\newtheorem{lemma}[theo]{Lemma}
\newtheorem{prop}[theo]{Proposition}

\newtheorem{defi}[theo]{Definition}

\theoremstyle{definition}

\newtheorem{rem}[theo]{Remark}

\newtheorem{exa}[theo]{Example}

\DeclareMathOperator{\diag}{diag}
\DeclareMathOperator{\Id}{Id}
\DeclareMathOperator{\Ker}{Ker}
\usepackage[utf8]{inputenc}
\usepackage{color}

 \def\mG{\mathsf{G}}
 
 \def\mV{\mathsf{V}}
 \def\mE{\mathsf{E}}

 \def\Gfun{\mathfrak{G}}

 \def\mv{\mathsf{v}}
 \def\me{\mathsf{e}}
 \def\mw{\mathsf{w}}

 \def\mf{\mathsf{f}}

\numberwithin{theo}{section}

\author{Delio Mugnolo}
\address{Delio Mugnolo, Institut f\"ur Analysis, Universit\"at Ulm, 89069 Ulm, Germany}
\email{delio.mugnolo@uni-ulm.de}

\title{Some remarks on the Krein--von Neumann extension of different Laplacians}

\thanks{The author is supported by the Land Baden--Württemberg in the framework of the \emph{Juniorprofessorenprogramm} -- research project on ``Symmetry methods in quantum graphs''. The author is grateful to Matthias Keller (Jena) for interesting discussions.}

\begin{document}

\begin{abstract}
We discuss the Krein--von Neumann extensions of three Laplacian-type operators -- on discrete graphs, quantum graphs, and domains. In passing we present a class of one-dimensional elliptic operators such that for any $n\in \mathbb N$ infinitely many elements of the class have $n$-dimensional null space.
\end{abstract}
\maketitle

\section{Introduction}

A classical theorem due to Krein (\cite{Kre47}) states that among all extensions of a densely defined, symmetric, positive semidefinite operator $A$ there are two exceptional operators that are extremal with respect to the natural order relation for unbounded self-adjoint operators: These are the Friedrichs (or ``hard'') extension $A_F$ -- the maximal one -- and what is nowadays commonly called Krein--von Neumann (or ``soft'') extension $A_K$ -- the minimal one. 

The Friedrichs extension $A_F$ turns out to agree in many relevant cases with the realization with Dirichlet boundary conditions.  The Krein--von Neumann realization $A_K$, however, is more delicate and in several respects less natural. This is partly due the null space of $A_K$, which by construction contains the null space of $A^*$ and is a such large -- indeed, in many cases even infinite dimensional. This feature typically jeopardizes the positivity of the generated semigroup and the Weyl asymptotics of its eigenvalues -- at least in their naivest form. In particular, the abstract Cauchy problem associated with $A_K$ may not be associated with a Markov process, even if the one associated with $A_F$ is. If one thinks of $A$ as a Laplacian, this is perhaps surprising in view of the properties commonly displayed by diffusion semigroups -- but is well in accordance with typical properties of Krein--von Neumann extensions, cf.~\cite[\S~2.3]{Fuk80}.

The theory of Krein--von Neumann extensions flourished in the 1980s (\cite{AloSim80,Fuk80,Gru83}). Recently the interest in the Krein--von Neumann extension, and in particular its associated boundary conditions, has arisen again, see e.g.~\cite{BroChr05,MakTse07,AshGesMit10,Zem11,Gru12}: An overview of recent results on these extensions and their connections with other problems in mathematical physics, along with a comprehensive list of references, can also be found in the survey article~\cite{AshGesMit13}.  

In this note short we recall some elementary features of this theory and apply them to discuss some properties of the Krein--von Neumann extensions of three Laplacian-type operators on network-like structures.

\section{General setting}

A partial order on the set of self-adjoint operators on a Hilbert space can be introduced as follows. 

\begin{defi}\label{defi:above}
Let $H$ be a Hilbert space. Let $a_1,a_2$ be two symmetric, bounded sesquilinear forms with domains $V_1,V_2$ that are both elliptic with respect to $H$. Denote by $A_1,A_2$ the associated operators on $H$. If $V_2\subset V_1$ and $a_1(x,x)\le a_2(u,u)$ for all $u\in V_2$, then $A_1$ is said to be \emph{smaller than or equal to} $A_2$ and one writes $A_1\le A_2$.
\end{defi}

Here and in the following we are adopting the terminology of~\cite[\S~VI.3.2]{DauLio88} to study linear operators on $H$; in particular, the positive semidefinite operator $A$ associated with a symmetric, bounded, elliptic sesquilinear form $a$ with domain $V$ is defined as usual as
\begin{equation}\label{eq:opassoc}
\begin{split}
{\rm Dom}(A)&:=\{u\in V:\exists v\in H:(v|w)_H=a(u,w)\ \forall w\in V\},\\
Af&:=g.
\end{split}\end{equation}
This operator is uniquely determined in view of the Lax--Milgram Lemma.

\begin{exa}\label{exa:special}
Let $\Omega\subset \mathbb R^d$ be a bounded open domain with 
Lipschitz boundary. Define $-\Delta_N$ (resp., $-\Delta_D$) as the operator acting on $L^2(\Omega)$ associated in the sense of~\eqref{eq:opassoc} with the form
\[
a:(u,v)\mapsto\int_\Omega \nabla u\nabla\bar v
\]
with domain
\[
V_N:=H^1(\Omega)\qquad \hbox{(resp., }V_D:=H^1_0(\Omega)\hbox{).} 
\]
Thus, $\Delta_N,\Delta_D$ are variationally defined realizations of the Laplacian -- indeed, the Laplacian with Neumann and Dirichlet boundary conditions, respectively. 

More generally, if $q\in \mathbb R$, then the symmetric, bounded sesquilinear form
\[
a_q:(u,v)\mapsto\int_\Omega \nabla u\nabla\bar v+q \int_{\partial\Omega} u_{|\partial\Omega}\bar{v}_{|\partial\Omega},\qquad u,v\in V_q:=H^1(\Omega).
\]
is elliptic with respect to $L^2(\Omega)$. The associated operator is the (variationally defined) operator $-\Delta_q$ -- we denote by $\Delta_q$ the Laplacian with Robin boundary conditions
\[
\frac{\partial u}{\partial \nu}+qu=0 \qquad \hbox{on }\partial \Omega.
\]
Then 
\[
-\Delta_{-q}\le -\Delta_N\le -\Delta_q\le -\Delta_D
\]
in the sense of Definition~\ref{defi:above}, whenever $0\le q$.

Also lying between $-\Delta_N,-\Delta_D$ -- but not comparable with the Robin Laplacians -- is any
operator $-\Delta_\omega$ associated with the form
\[
a_\omega:(u,v)\mapsto\int_\Omega \nabla u\nabla\bar v,\qquad u,v\in V_\omega:=\{w\in H^1(\Omega):w_{|\omega}=0\},
\]
whenever $\partial \Omega$ consists of two disjoint closed sets $\omega$ and $\partial\Omega\setminus \omega$ -- think e.g.\ of an annulus.
(In other words, $\Delta_\omega$ is the Laplacian with Dirichlet boundary conditions on a part $\omega$ of the boundary $ \partial\Omega$ of $\Omega$ and Neumann boundary conditions elsewhere.)

Now, $-\Delta_{q}$ is positive semidefinite if and only if $q\ge 0$, since
\[
u\mapsto \int_\Omega |\nabla u|^2
\]
does not define an equivalent norm on $H^1(\Omega)$. However, $-\Delta_N$ (i.e., $-\Delta_q$ for $q=0$) is in general not minimal (in the sense of Definition~\ref{defi:above}) among the positive semidefinite, selfadjoint extensions of the Laplacian defined on the space $C^\infty_c(\Omega)$ -- not even if $\Omega=(0,1)$, as we are going to see in Example~\ref{exa-deriv} below. 
\end{exa} 

The following summarizes two results obtained in~\cite{Kre47,AndNis70}, cf.~\cite[Chapters~13--14]{Sch12}.

\begin{theo}\label{thm:Kre47}
Let $A$ be a symmetric, positive semidefinite operator on $H$. Then $A$ has a self-adjoint extension if and only if the associated quadratic form is densely defined in $H$.

In this case, there exist precisely two extensions $A_K,A_F$ of $A$ such that
\begin{itemize}
\item $A_K,A_F$ are self-adjoint and positive semidefinite and
\item \emph{any} other self-adjoint, positive semidefinite extension $\tilde{A}$ of $A$ satisfies
\[
A_K\le \tilde{A}\le A_F.
\]
\end{itemize}
\end{theo}

The operators $A_K$ and $A_F$ are usually referred to as the \emph{Krein--von Neumann} (or \emph{soft}) \emph{extension} and the \emph{Friedrichs} (or \emph{hard}) \emph{ extension} of $A$, respectively.

\begin{exa}\label{exa-deriv}
Let us revisit the setting of Example~\ref{exa:special} by letting
\[
A:=-\Delta,\qquad {\rm Dom}(A):=C^\infty_c(\Omega),\qquad H:=L^2(\Omega).
\]
We are going to focus on the case 
\[
\Omega=(0,1).
\]
Then the Krein--von Neumann extension $A_K$ of $A$ is known. It is shown in~\cite[Example~5.1]{AloSim80}, cf.\ also~\cite[Example~14.14]{Sch12}, that 
$A_K$ is $-\frac{d^2}{dx^2}$ defined on the space of $H^2(0,1)$-functions with boundary conditions
\begin{equation}\label{eq:bckrein}
u'(1)=u'(0)=u(1)-u(0)\ ;
\end{equation}
equivalently, $A_K$ is the operator associated with the symmetric, bounded, elliptic sesquilinear form 
\begin{equation}\label{aalonsosim}
H^1(0,1) \ni (u,v)\mapsto \int_0^1 u'\overline{v'}\ dx-\left(\begin{pmatrix}1 & -1\\ -1 & 1 \end{pmatrix}\begin{pmatrix}u(1)\\ u(0)\end{pmatrix}\mid \begin{pmatrix}v(1)\\ v(0)\end{pmatrix}\right)\in \mathbb C\ ,
\end{equation}
cf.\ also~\cite{BobMug13}. As in~\cite[\S~5]{CarMug09}, a direct application of the Beurling--Deny conditions then shows that the associated semigroup is neither positive nor $L^\infty$-contractive, since neither of these properties is enjoyed by the semigroup 
\[
\exp\left(t\begin{pmatrix}1 & -1\\ -1 & 1 \end{pmatrix}\right)=\frac12 \begin{pmatrix}1+e^{2t} & 1-e^{2t}\\ 1-e^{2t} & 1+e^{2t} \end{pmatrix},\qquad t\ge 0.
\]
This also holds in an a more general setting, cf.~\cite[\S~2.3]{Fuk80}.

This characterization of the boundary conditions~\eqref{eq:bckrein} for the Krein--von Neumann extension of $A$ on $H=L^2(\Omega)$ for $\Omega=(0,1)$ has a pendant in the case where $\Omega$ is a bounded open domain of $\mathbb R^d$, under some mild assumption on the regularity of $\partial \Omega$, cf.~\cite{AshGesMit10b}.
\end{exa}

The possibility of explicitly describing the Krein--von Neumann extensions of a symmetric, positive semidefinite operator $A$ on a Hilbert space $H$, like in Example~\ref{exa-deriv}, is made possible by an approach based on symmetric forms, cf.~\cite[\S~2.3]{Fuk80}, which we present next for the sake of self-containedness: First of all take the closure of
\[
(u,v)\mapsto (Au|v)_H,\qquad u,v\in {\rm Dom}(A),
\]
to obtain a form $(a_F,V_F)$ (if $A$ is not yet self-adjoint). The associated operator is $A_F$, the Friedrichs extension of $A$. 
Then for all self-adjoint positive semidefinite extensions $\tilde{A}$ of $A$ the direct sum representation
\begin{equation}\label{eq:royd}
\tilde{V}=V_F \oplus \left(\Ker(\Id+A^*)\cap \left\{w\in H: \lim_{\lambda \nearrow 0}\left((\lambda\Id+A_F)^{-1}w\mid w\right) \emph{ exists in $\mathbb R$}\right\} \cap \tilde{V} \right)
\end{equation}
of the associated form domain $\tilde{V}$ holds. Furthermore, the Krein--von Neumann extension $A_K$ of $A$ is the operator associated with the form $a_K$ defined on by
\begin{equation}\label{krein0}
a_K(u,v):=\left\{
\begin{array}{ll}
a_F(u,v),\qquad &\hbox{if }u,v\in V_F,\\
\left(\lim\limits_{\lambda \nearrow 0}(\lambda \Id+A_F)^{-1} u\mid v\right),\qquad &\hbox{if }u,v\in \tilde{V}\ominus V_F,\\
- (u\mid v)_H,\qquad &\hbox{if }u\in V_F,\ v\in  \tilde{V}\ominus V_F,
\end{array}
\right.
\end{equation}
cf.~\cite[Lemma~2.3.2]{Fuk80} and following observations. The space $\tilde{V}$ may be much larger than $V_F$, and in particular it may happen that $\tilde{V}$ is not compactly embedded in $H$ even if $V_F$ is.

\begin{rem}
In the special case discussed in Example~\ref{exa:special},~\eqref{eq:royd} is an abstract version of the so-called \emph{Royden decomposition}, i.e., of the assertion that $H^1(\Omega)$ is the direct sum of $H^1_0(\Omega)$ and the space ${\rm Har}(\Omega)$ of (weakly) harmonic functions. (Observe that ${\rm Har}(\Omega)$ is a closed subspace of $H^1(\Omega)$ of dimension $2$ if $d=1$, and of infinite dimension for $d\ge 2$.)

This identity has been rediscovered again and again in different contexts, see e.g.~ \cite[Lemma~1.2]{Gre87}, \cite[Lemma~2.1]{BatBinDij05}, \cite[Thm.~3.6]{HaeKelLen12}, \cite[Thm.~2.5]{AreEls12}, or~\cite[Notes~I.2]{Woe00} for older references.)
\end{rem}

If additionally 
\[
A-\epsilon \Id \quad \hbox{is positive semidefinite for some }\epsilon>0,
\]
e.g.\ if $-A$ is self-adjoint, dissipative and injective and has compact resolvent, then the above construction can be refined to yield
that the symmetric form on $V_K:=V_F\oplus \Ker(A^*)$ associated with $A_K$ is simply given by
\begin{equation}\label{eq:kreinchar2}
a_K(u,v):=a_F(Pu,Pv)=
\left\{
\begin{array}{ll}
a_F(u,v)\qquad &\hbox{if }u,v\in V_F,\\
0 &\hbox{otherwise},
\end{array}
\right.
\end{equation}
where $P$ is the orthogonal projector of $V_K$ onto $V_F$. Indeed the Krein--von Neumann extension can be also characterized as follows, as a consequence of~\cite{Gru83} and~\cite[\S~2]{AshGesMit10}.

\begin{lemma}\label{lem:kreinchar}
Let $A$ be a symmetric, densely defined operator on $H$. If $A- \epsilon\Id$ is positive semidefinite for some $\epsilon>0$, then among all the self-adjoint positive semidefinite extensions of $A$ there exists exactly one whose domain contains $\Ker(A^*)$: This is precisely the Krein--von Neumann extension $A_K$ of $A$.
\end{lemma}

\begin{rem}\label{rem:inviewof}
Clearly, $A_K$ leaves $\Ker (A^*)$ and hence $H\ominus \Ker (A^*)$ invariant. Following~\cite[\S~2]{AshGesMit10} we call the part $\tilde{A}_K$ of $A_K$ in $H\ominus \Ker (A^*)$ the \emph{reduced Krein--von Neumann} extension of $A$. It has been proved in~\cite{Gru83,AshGesMit10} that (minus) the reduced Krein--von Neumann Laplacian is unitarily equivalent to a certain realization of the bi-Laplacian that arises in the so-called \emph{buckling problem} of elasticity theory.

The study of the reduced Krein--von Neumann extension was was initiated in~\cite[\S~5]{AloSim80}. The main motivation for this was the observation that $\tilde{A}_K$ has discrete spectrum if so does $A_F$ (remember that the form domain of $A_K$ may still be non-compactly embedded in $H$), and furthermore that the $k$-th eigenvalue of $\tilde{A}_K$ dominates the $k$-th eigenvalue of $A_F$, for each $k\in \mathbb N$.
\end{rem}

In the following sections we discuss different Laplacian-type operators whose Krein--von Neumann extensions seem not to have been considered in the literature so far.

\section{The discrete Laplacian}\label{sec:discretelapl}
Krein--von Neumann extension of matrices have been studied already in~\cite{BroChr05}, in the framework of the theory of Jacobi matrices. In this case we focus on graph Laplacians.

Let $\mG=(\mV,\mE)$ be a possibly infinite simple oriented graph. More precisely,
\begin{itemize}
\item $\mV$ is a set that is at most countable,
\item $\mE\subset \mV\times \mV$ and
\item for any  $\mv,\mw\in \mV$ one has $(\mv,\mv)\not \in \mE$ and $(\mv,\mw)\in \mE\Rightarrow (\mw,\mv)\not\in \mE$.
\end{itemize}
We refer to the elements of $\mV,\mE$ as \emph{nodes} and \emph{edges}, respectively; and to $\mv,\mw$ as the \emph{initial} and \emph{terminal endpoint} of the edge $\me=(\mv,\mw)$, respectively. For each simple oriented graph $\mathcal I=(\iota_{\mv\me})$ we can define the oriented incidence matrix of $\mG$ by
$${\iota}_{\mv \me}:=\left\{
\begin{array}{ll}
+1 & \hbox{if}~\mv~\hbox{is terminal endpoint of } \me,\\
-1 & \hbox{if}~\mv~\hbox{is initial endpoint of } \me,\\
0 & \hbox{otherwise},
\end{array}\right.\qquad \mv\in \mV,\me\in \mE.$$ 
We also assume for simplicity $\mG$ to be connected. 

Consider furthermore a weight function $\rho:\mE\to (0,\infty)$. In view of the known correspondence between $\mathcal I$ and the divergence operator of vector analysis, see e.g.~\cite{RigSalVig97,GraPol10}, the (possibly unbounded) operator
\[
\begin{split}
{\rm Dom}(\mathcal L)&:=c_{00}(\mV),\\
\mathcal L&:=\mathcal I\mathcal R\mathcal I^T,
\end{split}
\]
on $\ell^2(\mV)$ is called the \emph{discrete Laplacian} on $\mG$ with respect to the weight $\rho$: Here $\mathcal R:=\diag(\rho(\me))_{\me\in \mE}$ and $c_{00}(\mV)$ denotes the space of sequences on $\mV$ with finite support. The corresponding symmetric, bounded, elliptic sesquilinear form is 
\[
a:(f,g)\mapsto\left(\mathcal R\mathcal I^T f\mid \mathcal I^T g\right)_{\ell^2(\mE)},\qquad f,g\in c_{00}(\mV),
\]
which is densely defined in $\ell^2(\mV)$: By Theorem~\ref{thm:Kre47} it does have self-adjoint extension(s). The Friedrichs extension is obtained as the operator associated closing up 
\[
(f,g)\mapsto (\mathcal L f|g)_{\ell^2(\mV)},\qquad f,g\in c_{00}(\mV)\ .
\]
The operator $\mathcal L_F$ associated with such a closure has been thoroughly investigated in~\cite{KelLen12,HaeKelLen12}, where it is referred to as ``Dirichlet Laplacian''.

In the following we are always going to assume that
\[
\rho\in \ell^\infty(\mE)\ :
\]
Then clearly $\mathcal L$ is a bounded and hence self-adjoint operator on $\ell^2(\mV)$ provided $\mathcal I$ is a bounded operator from $\ell^2(\mE)$ to $\ell^2(\mV)$: By~\cite[Lemma~4.3]{Mug14} this latter condition is satisfied if  $\mG$ is \emph{uniformly locally finite}, i.e., if $\deg_\rho\in \ell^\infty(\mV),$ where 
\[
\deg_\rho(\mv) :=\sum_{\me\in \mE}|\iota_{\mv\me}|\rho(\me),\qquad \mv\in \mV,
\]
and in particular if $\mV$ is finite. 

But in the general case of $\deg_\rho \not\in \ell^\infty(\mV)$ there may exist several self-adjoint extensions. The maximal domain of the form $a$ is
$\{f\in \ell^2(\mV):a(f,f)<\infty\}$, i.e.,  the discrete Sobolev space
\[
w^{1,2}(\mV):=\{f\in \ell^2(\mV):\mathcal I^T f\in \ell^2(\mE)\},
\]
which is a separable Hilbert space with respect to the inner product
\[
(f\mid g)_{w^{1,2}}:=(f\mid g)_{\ell^2(\mV)}+ (\mathcal I^T f\mid \mathcal I^T g)_{\ell^2(\mE)}.
\]
The associated operator (in the sense of~\eqref{eq:opassoc}) of $a$ with this maximal domain is by~\cite[Thm.~2.2]{HaeKelLen12}
\[
\begin{split}
{\rm Dom}(\tilde{\mathcal L})&:=\{f\in \ell^2(\mV):\mathcal I\mathcal R{\mathcal I}^T\in \ell^2(\mV)\},\\
\tilde{\mathcal L}
&:=\mathcal I\mathcal R\mathcal I^T.
\end{split}
\]


It is known that the closure $w^{1,2}_0(\mV)$ of $c_{00}(\mV)$ in $w^{1,2}(\mV)$ does not necessarily agree with $w^{1,2}(\mV)$, much in analogy to what happens with usual Sobolev spaces on open subsets of the Euclidean space $\mathbb R^d$. By~\cite[Thm.~3.6]{HaeKelLen12}, 
\[
w^{1,2}(\mV)=w^{1,2}_0(\mV)\oplus \Ker(\mathcal L+\Id).
\]


Observe that $\tilde{\mathcal L}$ is a matrix with negative off-diagonal entries, hence one would naively expect any restriction of $\tilde{\mathcal L}$ to generate a positive semigroup. 
A class of self-adjoint extensions of $\mathcal L$ is characterized in~\cite[Thm.~5.2]{HaeKelLen12}, in dependence on the Markov property (or lack thereof) of the generated semigroup. In view of Theorem~\ref{thm:Kre47}, this characterization is possibly not exhaustive, and in particular a concrete example of a weighted graph $\mG$ and a discrete Laplacian (more precisely: of an extension of $\mathcal L$) that generates a non-submarkovian semigroup is presented in~\cite[Example~5.1]{HuaKelMas13}: It is currently not clear whether this operator from~\cite{HuaKelMas13} is the Krein--von Neumann extension of the discrete Laplacian constructed as in~\eqref{krein0}. 

\begin{rem}
In view of~\cite[Prop.~3.8]{Mug14}, $w^{1,2}(\mV)$ and hence $w^{1,2}_0(\mV)$ are compactly embedded in $\ell^2(\mV)$ if for every $\epsilon>0$ there are $\mv\in \mV$ and $r>0$ such that 
\begin{enumerate}[(i)]
\item $B_\rho(\mv,r):=\{\left\{\mw\in \mV:{\rm dist_\rho}(\mv_0,\mw)<r\right\}$ is a finite set and additionally
\item there holds
\[
\sum_{\mw\not\in B_\rho(\mv,r)} |f(\mw)|^2 <\epsilon^2
\]
for all $f$ such that $\|f\|_{w^{1,2}(\mV)}\le 1$.
\end{enumerate}
If these conditions are satisfied, then $\mathcal L_F$ has discrete spectrum, and hence so does the reduced Krein--von Neumann extension $\mathcal L_K$ of $\mathcal L$, cf.\ Remark~\ref{rem:inviewof}.
\end{rem}

\section{The quantum graph Laplacian}
Let $\mG$ be a simple oriented graph as in the previous section. Then, $\mG$ is turned into a \emph{metric} or \emph{quantum graph} $\Gfun$ by identifying each edge $\me$ with an interval $\left(0,\rho(\me)\right)$ and the initial or terminal endpoint $\mv$ of $\me$ with $0$ or $\rho(\me)$, respectively, cf.~\cite[Chapter~3]{Mug14} for a precise definition. One may then define a second derivative operator $\Delta_0$ on the space of smooth functions that have compact support on each interval, i.e.,
\begin{equation}\label{deltamit}
\Delta_0:(u_\me)_{\me\in \mE} \mapsto \left(\frac{d^2 u_\me}{dx^2}\right)_{\me\in \mE},\qquad {\rm Dom}(\Delta_0):=\prod_{\me \in \mE}C^\infty_c(0,\rho(\me)).
\end{equation}
For the sake of simplicity, let us in the following assume $\mG$ to be finite.

There are uncountably many self-adjoint positive semidefinite extensions of $-\Delta_0$ in the Hilbert space
\[
H:=L^2(\Gfun):=\prod_{\me \in \mE}L^2(0,\rho(\me)),
\]
cf.~\cite[\S~1.4]{BerKuc13}, but most of them will not be faithful to the original structure of $\mG$: That is, edges that are adjacent in $\mG$ may not necessarily be assigned any gluing conditions and, on the contrary, non-adjacent edges may possibly be. 

For this reason, let us rather focus on functions satisfying the continuity condition
\begin{equation}\tag{Cc}
u_\me(\mv)=u_\mf(\mv)=:u(\mv),\qquad \hbox{for all }\me,\mf\in \mE_\mv,\; \mv\in \mV,
\end{equation} 
(where $\mE_\mv$ denotes the set of edges one of whose endpoints is $\mv$) or, more formally: 
\begin{equation}\label{eq:moref}
\exists d\in \mathbb C^\mV \hbox{ such that }\left(\mathcal I^{+^T}d\right)_\me=u_\me(\rho(\me))\hbox{ and } \left(\mathcal I^{+^T}d\right)\me=u(0),\ \hbox{for all }\me\in \mE
\end{equation}
where $\mathcal I^+=(\iota^+_{\mv\me})$ and $\mathcal I^-=(\iota^-_{\mv\me})$ are the positive and negative parts of $\mathcal I=(\iota_{\mv\me})$, respectively. 

We thus study the second derivative operator $\Delta_{cont}$ defined as the formal extension of $\Delta_0$ to
\[
{\rm Dom}(\Delta_{cont}):=\left\{u\in \prod_{\me \in \mE}W^{2,2}(0,\rho(\me)): u \hbox{ satisfies $\rm(Cc)$ and }u'_\me(0)=u'_\me(\rho(\me))=0 \hbox{ for each }\me\in \mE\right\}.
\]
It is easy to see that $A:=-\Delta_{cont}$ is symmetric and positive semidefinite.
This domain incorporates ``too many'' boundary conditions, hence the operator $A$ cannot be self-adjoint. However, ${\rm Dom}(\Delta_{cont})$ contains $\prod_{\me\in \mE}C^\infty_c(0,\rho(\me))$, hence $A$ is densely defined and by Theorem~\ref{thm:Kre47} it admits self-adjoint extensions. These can be recovered by abstract extension theory. Admittedly, only a part of all possible self-adjoint extensions of $\Delta_0$ are found in this way, but on the other hand in this way we are sure that the domains of these extensions will contain the continuity condition, hence the connectivity of the graph will be respected.

Closing up the associated form 
\[
a:(u,v)\mapsto (Au|v)_H=(\Delta_{cont} u|v)_{L^2(\Gfun)},\qquad u,v\in {\rm Dom}(\Delta_{cont}),
\]
one finds the sesquilinear form
\[
(u,v)\mapsto \sum_{\me\in\mE}\int_0^{\rho(\me)}u'_\me \bar{v}'_\me,\qquad u,v\in W^{1,2}(\Gfun):=\left\{w\in \prod_{\me \in \mE}W^{1,2}(0,\rho(\me)): w \hbox{ satisfies $\rm(Cc)$}\right\}.
\]
Its associated operator is well-known in the literature: It is the formal extension of $\Delta_0$ to the domain
\[
\left\{w\in \prod_{\me \in \mE}W^{2,2}(0,\rho(\me)): w \hbox{ satisfies $\rm(Cc)$ and~$\rm(Kc)$}\right\},
\]
where
\begin{equation}\tag{Kc}
\partial_\nu u(\mv):=
\sum_{\me\in \mE}\iota^+_{\mv\me}u'_\me (1)-\sum_{\me\in \mE}\iota^-_{\mv\me}u'_\me (0) =0,\qquad \hbox{for all }\mv\in \mV,
\end{equation}
see e.g.~\cite[Lemma~2.3]{KraMugSik07}. (In other words, $(Kc)$ imposes that in each node the total incoming flow agrees with the total outgoing flow.) By construction we obtain the following.

\begin{prop}
The operator $A_F$ defined by
\[
\begin{split}
{\rm Dom}(A_F)&:=\left\{w\in \prod_{\me \in \mE}W^{2,2}(0,\rho(\me)): w \hbox{ satisfies $\rm(Cc)$ and~$\rm(Kc)$}\right\},\\
A_F &:(u_\me)_{\me\in \mE} \mapsto \left(\frac{d^2 u_\me}{dx^2}\right)_{\me\in \mE},
\end{split}
\]
is the Friedrichs extension of $A=-\Delta_{cont}$ on $L^2(\Gfun)$.
\end{prop}

Observe that $A_F$ is not injective -- indeed, its null space coincides with the null space of $\Delta_{cont}$, i.e., with the space of functions that are connected-componentwise constant.

Let us now turn to the Krein--von Neumann extension $A_K$ of $A$, which we might determine by means of~\eqref{krein0}. However, in this specific case it is easier and more enlightening to perform a direct computation. By~\cite[Thm.~1.4.4]{BerKuc13}, all self-adjoint extensions of $A=-\Delta_{cont}$ must satisfy additional boundary conditions
\begin{equation}\label{eq:questo}
\partial_\nu u(\mv)+\Lambda u(\mv)=0,\qquad \hbox{for all }\mv\in \mV,
\end{equation}
where $\Lambda$ is a self-adjoint operator acting on the Hilbert space $\ell^2(\mV)$.
It is not difficult to see that the associated quadratic forms are given by
\begin{equation}\label{eq:aform}
(u,v)\mapsto \sum_{\me\in\mE}\int_0^{\rho(\me)}u_\me' \bar v_\me'\ - \sum_{\mv,\mw\in \mV} \Lambda_{\mv\mw} u(\mv)\bar v(\mw),\qquad u,v\in W^{1,2}(\Gfun).
\end{equation}
Our goal is to find out for which $\Lambda$
\begin{equation}\label{eq:thissum}
\sum_{\me\in\mE}\int_0^{\rho(\me)}|u_\me'|^2 \ge \sum_{\mv,\mw\in \mV} \Lambda_{\mv\mw} u(\mv)\bar u(\mw)\qquad \hbox{for all }u\in W^{1,2}(\Gfun).
\end{equation}
Hölder's inequality yields
\[
\int_\alpha^\beta|w|^2\ge \frac{1}{\beta-\alpha} \left| \int_\alpha^\beta w\right|^2\qquad \hbox{for all }w\in L^2(\alpha,\beta),
\]	
for any two real numbers $\alpha<\beta$, and in view of its optimality the right hand side in~\eqref{eq:thissum} is made as small as possible if
\[
\sum_{\mv,\mw\in \mV} \Lambda_{\mv\mw} u(\mv)\bar u(\mw)=\sum_{\me\in\mE}\frac{1}{\rho(\me)}\left|\int_0^{\rho(\me)}u_\me'\right|^2.
\]
This is the case if and only if $\Lambda=\mathcal I\mathcal R^{-1}\mathcal I^T$\footnote{ We stress that $\mathcal I\mathcal R^{-1}\mathcal I^T$ is the discrete Laplacian of on $\mG$ with respect to the weight $\rho^{-1}$, whereas in Section~\ref{sec:discretelapl} we have considered the discrete Laplacian with respect to the weight $\rho$. We can think of weights $\rho,\rho^{-1}$ as resistances (proportional to a wire's length) and conductances  (inversely proportional to a wire's length), respectively. We need not care about realizations of $\mathcal I\mathcal R^{-1}\mathcal I^T$, since $\mV$ is finite and hence $\mathcal I\mathcal R^{-1}\mathcal I^T$ is bounded by assumption. }. Indeed, for this choice of $\Lambda$ and owing to the Fundamental Theorem of Calculus
\[
\sum_{\mv,\mw\in \mV} \Lambda_{\mv\mw} u(\mv)\bar u(\mw)=\sum_{\substack{\mv,\mw\in \mV\\ \mv\sim\mw}} \frac{1}{\rho\big((\mv,\mw) \big)}|u(\mw)- u(\mv)|^2=\sum_{\me\in\mE}\frac{1}{\rho(\me)}\left|\int_0^{\rho(\me)}u_\me'\right|^2,
\]
where we write $\mv\sim\mw$ whenever either edge $(\mv,\mw)$ or $(\mw,\mv)$ belongs to $\mE$.

Summing up, we have obtained the following.

\begin{prop}
The operator $A_K$ defined by
\[
\begin{split}
{\rm Dom}(A_K)&:=\left\{w\in \prod_{\me \in \mE}W^{2,2}(0,\rho(\me)): w \hbox{ satisfies $\rm(Cc)$ and~$\rm(KNc)$}\right\},\\
A_K&:(u_\me)_{\me\in \mE} \mapsto \left(\frac{d^2 u_\me}{dx^2}\right)_{\me\in \mE},
\end{split}
\]
where
\begin{equation}\tag{KNc}
\partial_\nu u(\mv) =\mathcal I\mathcal R^{-1}\mathcal I^T u(\mv)\qquad \hbox{for all }\mv\in \mV,
\end{equation}
is the Krein--von Neumann extension of $A=-\Delta_{cont}$ on $L^2(\Gfun)$.
\end{prop}

(In the trivial case of an unweighted graph that consists only of two adjacent nodes we recover the results in Example~\ref{exa-deriv}, as the matrix that appears in~\eqref{aalonsosim} is precisely the discrete Laplacian of this graph.)

\begin{prop}
The $C_0$-semigroup by $-A_F$ is Markovian, i.e., it is positive and contractive with respect to the $\infty$-norm. The $C_0$-semigroup by $-A_K$ is neither positive, nor contractive with respect to the $\infty$-norm.
\end{prop}

\begin{proof}
The Markov property of the semigroup generated by $\Delta_{cont}=-A_F$ has been proved in~\cite{KraMugSik07}. 

Just like in  Example~\ref{exa-deriv}, the semigroup generated by $-A_K$ cannot be positive in view of the formula~\eqref{eq:aform}, because $\mathcal I\mathcal R^{-1}\mathcal I^T$ has negative off-diagonal entries, so that $\left(\exp(t\mathcal I\mathcal R^{-1}\mathcal I^T)\right)_{t\ge 0}$ is not positive. Likewise, it is not contractive with respect to the $\infty$-norm because neither is $\left(\exp(t\mathcal I\mathcal R^{-1}\mathcal I^T)\right)_{t\ge 0}$, by~\cite[Lemma~6.1]{Mug07}.
\end{proof}

An alternative way of proving non-positivity of the semigroup generated by $-A_K$ is to observe that the null space of $A_K$ is higher-dimensional, which is not compatible with positivity of a $C_0$-semigroup in view of a version of the Perron--Frobenius theorem, cf.~\cite[Thm.~C.III.3.12]{Nag86}. 

\begin{prop}
The null space of $A_K$ has dimension $|\mV|$.
\end{prop}

\begin{proof}
Let $u\in {\rm Dom}(A_K)$. If $A_K u=0$, then $u$ has to be edgewise affine, i.e.,
\begin{equation}\label{eq:affine}
u_{\me}(x)=a_\me  x+b_\me,\qquad x\in (0,\rho(\me)),\ \me \in \mE,
\end{equation}
for some vectors $a,b\in \mathbb C^\mE$. Let us show that the space of affine functions that belong to ${\rm Dom}(A_K)$ has dimension $|\mV|$, i.e., that only $2|\mE|-|\mV|$ among the $2|\mE|$ entries of $a,b$ are determined by the node conditions $\rm(Cc)$ (or equivalently~\eqref{eq:moref}) and $\rm(KCn)$. To begin with, we remark that the matrix 
\[
\begin{pmatrix}
\mathcal I^{+^T}\\
\mathcal I^{-^T}
\end{pmatrix}
\]
is injective: Indeed, take $x\in \mathbb C^\mV$ and observe that the above matrix maps $x$ into a vector in $\mathbb C^{\mE\times \mE}$ whose $\me$-th (resp., $(|\mE|+\me)$-th) entry is the value $x_\me$ (resp.,\ $x_{|\mE|+\me}$) attained by $x$ in the node of $\mG$ that is terminal (resp., initial) endpoint of $\me$. Because each node is of course (terminal or initial) endpoint of at least one edge (otherwise $L^2(\Gfun)=\emptyset$), this implies that $x_\mv=0$ for each $\mv\in \mV$, i.e., $x=0$. Consequently, the above matrix has rank $|\mV|$ and the claim follows.

On the other hand, condition $\rm(KNc)$ can be equivalently written as
\[
\mathcal I^+ u'(\underline{\rho})-\mathcal I^- u'(0)=\mathcal I\mathcal R^{-1}\mathcal I^T d\ ,
\]
where $u'(\underline{\rho}):=(u'_\me(\rho(\me)))_{\me\in \mE}$ and $d\equiv \left(u(\mv)\right)_{\mv\in \mV}\in \mathbb C^\mV$ is the vector of nodal values that appears in~\eqref{eq:moref}. Now, by~\eqref{eq:moref} we obtain
\[
\mathcal I\mathcal R^{-1}\mathcal I^T d=\mathcal I\mathcal R^{-1}\left(\mathcal I^+ d-\mathcal I^- d)=\mathcal I\mathcal R^{-1}(u(\underline{\rho})-u(0)\right),
\]
where $u(\underline{\rho}):=(u_\me(\rho(\me)))_{\me\in \mE}$. But for functions of the form~\eqref{eq:affine} 
\[
u'_\me(\rho(\me))=u'_\me(0)=a_\me\quad\hbox{and}\quad u_\me(\rho(\me))-u_\me(0)=a_\me \rho(\me)\qquad \hbox{for all }\me\in \mE,
\]
i.e., condition $\rm(KNc)$ turns out to be void. This completes the proof.
\end{proof}

\section{Wentzell-type boundary conditions}

Let $\Omega\subset \mathbb R^d$ be a bounded open domain with $(d-1)$-dimensional Lipschitz boundary. Let us consider again the operator $A$ introduced in Example~\ref{exa-deriv}, i.e.,
\begin{equation}\label{deltamit2}
-\Delta\quad \hbox{with domain}\quad C^\infty_c(\Omega),
\end{equation}
which satisfies the assumptions of Theorem~\ref{thm:Kre47} with the respect to $H=L^2(\Omega)$. 
Consider the isomorphism
\[
\Phi:C(\overline{\Omega})\ni u\mapsto\begin{pmatrix} u\\ u_{|\partial \Omega}\end{pmatrix}\in C(\overline{\Omega})\times C(\partial\Omega).
\]
Now, the isomorphic image of $A=-\Delta$ under $\Phi$ is symmetric and positive semidefinite in the larger Hilbert space 
\[
\mathbb H:=L^2(\Omega)\times L^2(\partial\Omega)\ ,
\]
too. However, its domain
\[
\Phi\left(C^\infty_c(\Omega)\right)=\left\{\begin{pmatrix} u\\ f\end{pmatrix}\in C^\infty_c(\Omega)\times C^\infty(\partial \Omega),\ u_{|\partial \Omega}=f \right\}
\]
is not dense in $L^2(\Omega)\times L^2(\partial\Omega)$, and indeed there are several, mutually not comparable closed operators on $L^2(\Omega)\times L^2(\partial\Omega)$ whose domain contains $\Phi(C^\infty_c(\Omega))$. However, 
\[
{\rm Dom}(\mathbb A):=\left\{\begin{pmatrix} u\\ f\end{pmatrix}\in C^\infty(\overline\Omega)\times C^\infty(\partial \Omega),\ u_{|\partial \Omega}=f,\ \frac{\partial u}{\partial \nu}=0 \right\}
\]
is indeed dense in $\mathbb H$, and for all $\eta_1,\eta_2\ge 0$ the operator
\[
\mathbb A:=-\begin{pmatrix}\Delta & 0\\0 & \eta_1\Delta_{\partial \Omega}-\eta_2\Id \end{pmatrix} \quad\hbox{with domain}\quad {\rm Dom}(\mathbb A)
\] 
is  symmetric and positive semidefinite.

\begin{rem}
The Laplace--Beltrami operator $-\Delta_{\partial \Omega}$ with domain $C^\infty(\partial \Omega)$ is in its own right symmetric and positive semidefinite, cf.~\cite[Chapter~5]{Dav89}: In fact, it is essentially self-adjoint and its closure is associated with a Dirichlet form, hence it generates a sub-Markovian semigroup on $L^2(\partial \Omega)$. 
Observe that the abstract Cauchy problem associated with $-\mathbb A$ is equivalent to the initial-value problem for 
\[
\left\{
\begin{split}
\frac{\partial u}{\partial t}(t,x)&=\Delta u(t,x),\qquad &t\ge 0,\ x\in \Omega,\\
\frac{\partial u}{\partial t}(t,z)&=\eta_1\Delta_{\partial \Omega} u(t,z)-\eta_2 u(t,z),\qquad & t\ge 0,\ z\in \partial\Omega,
\end{split}
\right.
\]
Taking the trace on $\partial \Omega$ of the first equation and plugging it into the second one, we obtain 
\[
\Delta u(t,z)=\eta_1\Delta_{\partial \Omega} u(t,z)-\eta_2 u(t,z),\qquad  t\ge 0,\ z\in \partial\Omega,
\]
a class of boundary conditions the study of which goes back to~\cite{Ven60}: For $\eta_1=\eta_2=0$ we recover the classical Wentzell boundary conditions studied by Feller already in the early 1950s, cf.~\cite[\S~VI.5]{EngNag00} and references therein.
\end{rem}

By the Gauß--Green formulae the quadratic form associated with $\mathbb A$ is
\[
a:\left(\begin{pmatrix}u\\ u_{|\partial \Omega}\end{pmatrix},\begin{pmatrix}v\\ v_{|\partial \Omega}\end{pmatrix} \right)\mapsto \int_\Omega \nabla u\cdot \overline{\nabla v} \ dx+\eta_1 \int_{\partial \Omega} \nabla u\cdot \overline{\nabla v} \ d\sigma-\eta_2 \int_{\partial \Omega} u \overline{v} \ d\sigma,
\]
with form domain ${\rm Dom}(\mathbb A)$, where $d\sigma$ denotes the surface measure of $\partial \Omega$. This form is closable and its closure is the form $a_F$ that acts just as $a$ does, defined on the form domain 
\[
\mathbb V_F:=\left\{
\begin{pmatrix} u\\ f\end{pmatrix}\in H^1(\Omega)\times D_2: \ u_{|\partial \Omega}=f \right\},
\]
where 
\[
D_2:=\left\{
\begin{split}
L^2(\partial\Omega)\quad &\hbox{if }\eta_1=0,\\
H^1(\partial \Omega)\quad &\hbox{if }\eta_1>0.
\end{split}
\right.
\]
Following the computations performed in~\cite{CocFavGol08,VazVit11,Mug14b} we can determine the operator associated with $a_F$.

\begin{prop}
The operator $\mathbb A_F$ defined by
\[
\begin{split}
{\rm Dom}(\mathbb A_F)&:=\left\{
\begin{pmatrix} u\\ f\end{pmatrix}\in H^1(\Omega)\times D_2,\ \Delta u\in L^2(\Omega),\ \eta_1\Delta_{\partial \Omega}f\in L^2(\partial\Omega), \ u_{|\partial \Omega}=f \right\},\\
\mathbb A_F&:=-\begin{pmatrix}\Delta & 0\\-\frac{\partial}{\partial \nu} & \eta_1\Delta_{\partial \Omega}-\eta_2\Id \end{pmatrix},
\end{split}
\]
is the Friedrichs extension of $\mathbb A$ on $\mathbb H$.
\end{prop}

We are now going to determine the Krein--von Neumann extension of $\mathbb A$. 
In the following we assume for the sake of simplicity that 
\[
\eta_2>0\ . 
\]

\begin{lemma}
If $\eta_2>0$, then the symmetric operator $\mathbb A-\epsilon \Id$ with domain ${\rm Dom}(\mathbb A)$ is positive semidefinite on $\mathbb H$ for some $\epsilon>0$.
\end{lemma}

\begin{proof}
We already know that $\mathbb A$ is positive semidefinite. In order to prove the claim it suffices to check that $\mathbb A_F$ is injective and has compact resolvent.

If $u\in {\rm Dom}(\mathbb A)$ with $\mathbb A u=0$, then 
\[
0=\int_\Omega |\nabla u|^2\ dx+\eta_1 \int_{\partial \Omega} |\nabla u|^2 \ d\sigma-\eta_2 \int_{\partial \Omega} |u|^2 \ d\sigma,
\]
hence $u$ is constant on $\overline\Omega$ and in fact it has to vanish identically because $\eta_2>0$. Furthermore, $\mathbb V_F$ is compactly embedded in $\mathbb H$: This has been observed in~\cite{AreMetPal03} in the case of $\eta_1=0$ and follows from the continuous embedding of the form domain in $H^1(\Omega)\times H^1(\partial \Omega)$ if $\eta_1>0$.
\end{proof}

We can thus apply Lemma~\ref{lem:kreinchar} and in particular~\eqref{eq:kreinchar2}. A direct computation shows that the adjoint $\mathbb A^*$ of $\mathbb A$ is given by
\[
\mathbb A^*=-\begin{pmatrix}\Delta & 0\\-\frac{\partial }{\partial \nu} & \eta_1\Delta_{\partial \Omega}-\eta_2\Id \end{pmatrix}
\] 
with domain
\[
\left\{\begin{pmatrix} u\\ f\end{pmatrix}\in L^2(\Omega)\times L^2(\partial\Omega),\ \Delta u\in L^2(\Omega),\ \frac{\partial u}{\partial \nu} \in L^2(\partial\Omega),\ \eta_1\Delta_{\partial \Omega}f\in L^2(\partial\Omega) \right\},
\]
so that its null space is given by
\[
\begin{split}
\Ker \mathbb A^*&=\left\{\begin{pmatrix} u\\ f\end{pmatrix}\in L^2(\Omega)\times L^2(\partial\Omega),\ \Delta u=0,\ \frac{\partial u}{\partial \nu}=\eta_1\Delta_{\partial \Omega}f-\eta_2 f\in L^2(\partial\Omega) \right\}\\
&=\{u\in L^2(\Omega):\Delta u=0\}\times \{f\in L^2(\partial\Omega):\eta_1\Delta_{\partial \Omega}f-\eta_2 f\in L^2(\partial\Omega) \}\\
&=\{u\in L^2(\Omega):\Delta u=0\}\times \{f\in L^2(\partial\Omega):\eta_1\Delta_{\partial \Omega}f\in L^2(\partial\Omega) \}\\
&={\rm Har}(\Omega)\times {\rm Dom}(\eta_1\Delta_{\partial \Omega}),
\end{split}
\]
(recall that by assumption $\eta_1\ge 0$ and therefore $\eta_1\Delta_{\partial \Omega}-\eta_2\Id $ is bijective from the domain of $\Delta_{\partial \Omega}$ to $L^2(\partial \Omega)$). We finally consider
\[
\mathbb V_K:=\mathbb V_F\oplus \Ker \mathbb A^*
\]
and denote by $\mathbb P$ the orthogonal projector of $\mathbb V_K$ onto $\mathbb V_F$. We henceforth study the quadratic form $a$ defined by
\[
a_K({\mathbb u},{\mathbb v}):=a_F({\mathbb{Pu}},{\mathbb{Pv}})\qquad \mathbb u,\mathbb v\in \mathbb V_K.
\]
A direct computation yields the operator associated with $a_K$ and we obtain the following.


\begin{theo}\label{theo:determkrein}
The operator $\mathbb A_K$ defined by
\[
\begin{split}
{\rm Dom}(\mathbb A_K)&:=\left\{
\begin{pmatrix} u\\ f\end{pmatrix}\in H^1(\Omega)\times D_2,\ \Delta u\in L^2(\Omega),\ \frac{\partial Pu}{\partial \nu}\in L^2(\partial\Omega),\ \eta_1\Delta_{\partial \Omega}f\in L^2(\partial\Omega), \ u_{|\partial \Omega}=f \right\},\\
\mathbb A_K&:=-\begin{pmatrix}\Delta & 0\\-\frac{\partial P}{\partial \nu} & \eta_1\Delta_{\partial \Omega}-\eta_2\Id \end{pmatrix},
\end{split}
\]
where $P$ is the orthogonal projector of $H^1(\Omega)$ onto $H^1(\Omega)\ominus {\rm Har}(\Omega)=H^1_0(\Omega)$, is the Krein--von Neumann extension of $\mathbb A$ on $\mathbb H$.
\end{theo}

\begin{rem}
The Dirichlet-to-Neumann operator 
\[
{\rm D\!\!N}:H^1(\partial\Omega)\to L^2(\partial\Omega)
\]
is a selfadjoint, positive semidefinite pseudo-differential operator of order 1 defined by
\[
{\rm D\!\!N} f:=-\frac{\partial u}{\partial \nu}
\]
whenever there exists $u\in H^1(\Omega)$ such that
\begin{equation}\label{bvp-dton}
\left\{
\begin{split}
\Delta u&=0\qquad \hbox{in }\Omega,\\
u&=f\qquad \hbox{on } \partial \Omega.
\end{split}
\right.
\end{equation}
(This definition has been generalized to so-called \emph{quasi-convex} domains in~\cite[\S\S~5--6]{AshGesMit13}.)
But by definition
\[
-{\rm D\!\!N} u_{|\partial \Omega}=\frac{\partial (\Id-P)u}{\partial \nu},
\]
where $P$ is the orthogonal projector of $H^1(\Omega)$ onto $H^1(\Omega)\ominus {\rm Har}(\Omega)$, so that an equivalent representation of $\mathbb A_K$ is
\[
\mathbb A_K=-\begin{pmatrix}\Delta & 0\\-\frac{\partial }{\partial \nu} & -{\rm D\!\!N}+\eta_1\Delta_{\partial \Omega}-\eta_2\Id \end{pmatrix}.
\]
Thus, the parabolic problem associated with $-\mathbb A_K$ is a heat equation with dynamic boundary conditions
\[
\frac{\partial u}{\partial t}(t,z)=-\frac{\partial u}{\partial \nu}(t,z) -{\rm D\!\!N}u(t,z)+\eta_1\Delta_{\partial \Omega}u(t,z)-\eta_2u(t,z),\qquad t\ge 0,\ z\in \partial \Omega,
\]
which is tightly related to
\[
\Delta u(t,z)+\frac{\partial u}{\partial \nu}(t,z) -(-\Delta_{\partial \Omega})^\frac12 u(t,z)-\eta_1\Delta_{\partial \Omega}u(t,z)+\eta_2u(t,z)=0,\qquad t\ge 0,\ z\in \partial \Omega,
\]
studied in~\cite[Exa.~5.9]{Pos13} -- since the Dirichlet-to-Neumann operator agrees with $-(-\Delta_{\partial \Omega})^\frac12$ up to a lower order perturbation whenever $\partial\Omega$ is smooth enough, cf.~\cite[Prop.~C.1, pag.~453]{Tay96}.
\end{rem}

\bibliographystyle{alpha}
\bibliography{../../referenzen/literatur}

\end{document}